\documentclass[preprint,12pt]{elsarticle}

\usepackage{enumerate}
\usepackage{amssymb}
\usepackage{lipsum}
\usepackage[a4paper, total={6.0in, 8.4in}]{geometry}
\usepackage{booktabs}
\usepackage{amsmath,amssymb,url}
\usepackage{enumitem} 
\usepackage{graphics} 
\usepackage{array}
\usepackage{dsfont}
\usepackage[all]{xy}
\usepackage{amsthm}
\usepackage{tikz}

\numberwithin{equation}{section}
\newtheorem{theorem}{Theorem}[section]
\newtheorem{lemma}[theorem]{Lemma}
\newtheorem{proposition}[theorem]{Proposition}
\newtheorem{corollary}[theorem]{Corollary}

\newtheorem{problem}[theorem]{Problem}

\newtheorem{definition}[theorem]{Definition}

\newtheorem{remark}[theorem]{Remark}
\newtheorem{example}[theorem]{Example}

\newcommand{\dn}{\mathord{\downarrow}\hspace{0.05em}}
\newcommand{\up}{\mathord{\uparrow}\hspace{0.05em}}

\newcommand{\dda}{\mathord{\Uparrow}\hspace{0.05em}}
\newcommand\blfootnote[1]{%
\begingroup
\renewcommand\thefootnote{}\footnote{#1}%
\addtocounter{footnote}{-1}%
\endgroup
}

\journal{Topology and its Applications}

\begin{document}

\begin{frontmatter}



\title{Applications of sequential spaces to Smyth power spaces and sober spaces}


\author{Zhengmao He}
\address{School of Sciences, Southwest Petroleum University,Chengdu 610500, China}
\begin{abstract} Recently, in the study of topological properties in Domain theory, the first countable spaces and Fr\'echet-Urysohn spaces receive close attention and are frequently employed. In this paper, we focus on sequential spaces weaker than Fr\'echet-Urysohn spaces. The main results are:

(1) For a well-filtered space $X$, the Scott topology and the upper vietoris topology on $K(X)$ coincide if $P_{S}(X)$ is a sequential space;

(2) every $\omega$-well-filtered coherent $d$- space $X$ is sober when $X\times X$ is a sequential space.
\end{abstract}
\begin{keyword} well-filtered space; Scott topology; upper vietoris topology; sequential space; sober space\\
\vspace*{0.2cm}
\end{keyword}


\end{frontmatter}
\blfootnote{This work is supported by the National Natural Science Foundation of China (Grant no.12471438) and the Sichuan Science and Technology Program (Grant No. 2026NSFSC0784).}
\blfootnote{E-mail address: hezhengmaomath@163.com (Z.M.He).}


\section{Introduction}
\label{}

Domain theory was initiated by Dana Scott to construct denotational models for the untyped \(\lambda\)-calculus. By introducing partial orders to model information ordering, domains provide mathematical spaces for higher-order computations, where higher-order function spaces are again domains. Three major powerdomain constructions (Smyth powerdomain, Hoare powerdomain and Plotkin powerdomain) model angelic, demonic and convex nondeterminism for concurrent and nondeterministic programs. For a $T_{0}$ space $X$, its Smyth powerdomain $K(X)$ consists of all nonempty compact saturated subsets of $X$, ordered by the reverse inclusion. On
$K(X)$, there are two important topologies, namely the Scott topology and the upper vietoris topology. The Vietoris topology is a highly useful topology with many desirable properties. For example, a base space $X$ is sober (respectively, well-filtered) if and only if its
$K(X)$ equipped with the upper vietoris topology is sober (respectively, well-filtered)(\cite{GG03,EEE20}). The upper vietoris topology is contained in the Scott topology When $X$ is a well-filtered space(\cite{GG03}). It is natural to investigate the conditions under which these two topologies on $K(X)$ coincide when $X$ is a well-filtered space. In (\cite{EEE21}), Schalk was the first to establish that the two topologies on $K(X)$  coincide whenever $X$ is a locally compact sober space. Subsequently, Brecht and Kawai extended this result to the case where $X$ is a second-countable sober space(\cite{EEF1}). Recently, it has been observed that the notions of first countability and the Fr\'echet-Urysohn property are particularly useful in the study of topological properties of $T_{0}$ spaces, leading to several noteworthy results. In particular, every first-countable well-filtered space is sober(\cite{EEF22}); the well-filtered reflection of a first-countable space coincides with its sobrification(\cite{EEF25}); and an
$\omega$-well-filtered space $X$ is sober whenever its sobrification is Fr\'echet-Urysohn(\cite{EEF24}). Furthermore, if the Smyth power space of every well-filtered space $X$ is first-countable (respectively, Fr\'echet-Urysohn), then the Scott topology on
$K(X)$ coincides with the upper vietoris topology(\cite{EEF24,EEF30}). In this paper, we proceed to investigate conditions under which the two topologies on $K(X)$ coincide. Sequential spaces play a central role in our investigation. In Section 3, we show that, for a well-filtered space $X$, if its Smyth power space is sequential, then the Scott topology on $K(X)$ coincides with the Upper vietoris topology.

Sober spaces establish important connections among point-set topology, order-theoretic structures, and domain-theoretic semantics. Recently, many significant results have been obtained in the study of sober spaces. Miao,Xi,Jia,Li and Zhao present two Scott sober dcpos whose product is not sober for the Scott topology(\cite{EEF31}). Subsequently, Xi, Sheng and Zhao proved that the complete Boolean algebra of all regular open subsets of the real line is not sober with respect to the Scott topology(\cite{EEF32}). These counterexamples offer a new perspective on the conditions under which well-filtered spaces are sober. In 2025, we established that a $\omega$-well-filtered coherent $d$-space $X$ is sober whenever its product $X\times X$ is a Fr\'echet-Urysohn space(\cite{EEF33}). In section 4, we prove that the conclusion remains valid if the assumption that $X\times X$ is a Fr\'echet-Urysohn space is weakened to the assumption that $X\times X$ is a sequential space.

\section{Preliminaries}
\label{}

\quad Let $P$ be a poset and $A\subseteq P$. Set $$\up A=\{x\in P\mid \exists\ a\in A, a\leq x \}$$ and $$\dn A=\{x\in P\mid \exists\ a\in A, x\leq a\}.$$
For every $x\in P$, we write $\dn\{x\}$ as $\dn x$ and $\up\{x\}$ as $\up x$. A nonempty subset $D\subseteq P$ is directed if $\forall\ a,b\in D$, $\up a\cap \up b \cap D\neq\emptyset$. Dually, a nonempty subset $D\subseteq P$ is filtered if $\forall\ a,b\in D$, $\dn a\cap \dn b \cap D\neq\emptyset$.

\quad A poset $P$ is called a {\em directed complete poset} ({\em dcpo}, for short) if each
directed set of $P$ has a supremum.

\quad Let $P$ be a poset and $x, y\in P$. We say that $x$ is {\em way below} $y$, in symbols $x\ll y$, iff for each directed subset $D\subseteq P$ for which $\bigvee D$ exists, $y\leq\bigvee D$ implies $\up x\cap D\neq\emptyset$.\ If $\dda a=\{b\in P\mid b\ll a\}$ is a directed set and $\bigvee \dda a=a$\ for every \ $a\in P$, we call $P$ a continuous poset.

\quad Let $P$ be a poset. A subset $U\subseteq P$ is called {\em Scott open} (see \cite{GG03,JG13})if

 (i) $U=\up U$, and

 (ii) for each directed subset $D$, $\bigvee D\in U$ implies
$D\cap U\neq \emptyset$, whenever $\bigvee D$ exists.\\
All Scott open sets of $P$ form the Scott topology $\sigma(P)$. We write $\Sigma P$ for $(P,\sigma(P))$.

\quad Let $P,Q$ be dcpos. Then the mapping $f:\Sigma P\longrightarrow \Sigma Q$ is continuous if and only if for each directed set $D\subseteq P$, $f(\bigvee D)=\bigvee\limits_{d\in D}f(d)$.

\quad Let $X$ be a topological space. A nonempty subset $F\subseteq X$ is called {\em irreducible}, if for closed sets $A, B\subseteq X$, $F\subseteq A\cup B$ implies $F\subseteq A$ or $F\subseteq B$. A topological space $X$ is {\em sober} if for every irreducible closed set $A$, there exists a unique $x\in X$ such that $A=\overline{\{x\}}$.

\quad Given a $T_{0}$ space $X$, the {\em specialization order} $\leq$ on $X$ is defined by
$$x\leq y \Longleftrightarrow x\in cl(\{y\}).$$
Unless otherwise stated, throughout
the paper, whenever an order-theoretic concept is mentioned in the context of a $T_{0}$ space $X$, it is to be interpreted with respect to the specialization order on $X$.

\quad Given a topological space $X$, the symbols $\mathcal{O}(X)$ represents the lattice of all open subsets of $X$ ordered by the inclusion.

\quad A topological space $(X,\tau)$ is a {\em $d$-space} if $X$ is a dcpo and for every $U\in\tau$, $U$ is Scott open with respect to the specialization order of $X$.

\quad A subset $A$ of a topological space $X$ is {\em saturated} if $A=\up A$. A topological space $X$ is {\em well-filtered} ({\em $\omega$-well-filtered}) if for each filtered (countably filtered) family $\mathcal{F}$ of compact saturated subsets of $X$ and each open set $U$ of $X$, $\bigcap \mathcal{F}\subseteq U$ implies $F\subseteq U$ for some $F\in\mathcal{F}$.

 \quad For a topological space $X$, we shall use $K(X)$ to denote the poset of all nonempty
compact saturated subsets of $X$ with the reverse inclusion order. The upper Vietoris topology on $K(X)$ is the topology that has $\{\Box U\mid U\in\mathcal{O}(X)\}$ as a base, where $\Box U=\{K\in\mathcal{Q}(X)\mid K\subseteq U\}$. The upper vietoris topological space (or called Smyth powerspace) is denoted by $P_{S}(X)$. Then the specialization order of the upper space $P_{S}(X)$ is the reverse inclusion order.

\quad A topological space $X$ is {\em coherent} if for each pair $K,Q\in K(X)$, $K\cap Q\in K(X)$. A topological space $X$ is {\em locally compact} if for each $U\in \mathcal{O}(X)$ and for each $x\in U$, there is an open set $V$ and a compact subset $K$ such that $x\in V\subseteq K\subseteq U$. A topological space $X$ is said to be  {\em core-compact } if  $\mathcal{O}(X)$ is continuous. It has been known that every locally compact space is core-compact.

\quad A topological space $X$ is a {\em retract} of a topological space $Y$ if there are two continuous maps $f:X\longrightarrow Y$ and $g:Y\longrightarrow X$ such that $g\circ f=id_{X}$.

\section{Applications of sequential spaces to Smyth power spaces}\label{sec:fm}

In this section, using the sequential spaces, we will present a sufficient condition that for a topological space $X$, the upper vietoris topology and the Scott topology on $K(X)$ coincide.

\begin{definition}{\rm (see \cite{RE77}) A topological space $X$ is a Fr\'echet-Urysohn space if for every subset $A$, $x\in \overline{A}$ implies there is a sequence $\{x_{n}\}_{n\in\omega}$ contained in $A$ converging to $x$.}
\end{definition}

\begin{definition}{\rm (see \cite{RE77}) Let $X$ be a topological space. We say $A\subseteq X$ is a sequential closed set if for each sequence $\{x_{n}\}_{n\in\mathbb{N}}$ in $A$, $\{x_{n}\}_{n\in\mathbb{N}}$ converges to $x$ implies $x\in A$. $X$ is called a sequential space if every sequential closed set is closed in $X$.}
\end{definition}

A subset $U$ of a topological space $X$ is sequential open if $X\setminus U$ is sequential closed. Furthermore, $U\subseteq X$ is sequential open iff for every sequence $\{x_{n}\}_{n\in\mathbb{N}}$ converging to $x\in U$, there is a $n_{0}$ such that $\{x_{n}\mid n_{0}\leq n\}\subseteq U$. A space is sequential space iff every sequential open set is open.

\begin{proposition} {\rm Let $X$ be a topological space. If $P_{S}(X)$ is a sequential space, then $X$ is a sequential space.}
\end{proposition}

\begin{proof} Let $A\subseteq X$ be a sequential closed set of $X$ and $x\in \overline{A}$. Then we have $$\up x\in \overline{\{\up a\mid a\in A\}}_{P_{S}(X)}.$$ Since $P_{S}(X)$ is a sequential space, there is a sequence $\{a_{n}\}_{n\in\mathbb{N}}$ in $A$ such that $\{\up a_{n}\}_{n\in\mathbb{N}}$ converges to $\up x$.

Claim: $\{a_{n}\}_{n\in\mathbb{N}}$ converges to $x$.

Let $U\in\mathcal{O}(X)$ and $x\in U$. Then we have $\up x\in\Box U$. So there is a $n_{0}$ such that $\{\up a_{m}\mid n_{0}\leq m\}\subseteq\Box U$ because $\up x\in \overline{\{\up a\mid a\in A\}}_{P_{S}(X)}$. This implies that $\{a_{m}\mid n_{0}\leq m\}\subseteq U$. Thus, $\{a_{n}\}_{n\in\mathbb{N}}$ converges to $x$.

As $A$ is sequential closed, $x\in A$. Therefore, $A$ is closed in $X$ and $X$ is a sequential sapce. \end{proof}

\begin{lemma} {\rm (see \cite{RE77}) Let $X$ be a sequential space and $A\subseteq X$ a closed set. Then the closed subspace $A$ is a sequential space.}
\end{lemma}

The following Example 3.5 demonstrates that converse of Proposition 3.3 is not valid.

\begin{example} {\rm Let $X_{0}$ be a uncountable set and $\infty\not\in X_{0}$. Set $X=X\cup\{\infty\}$ and the topology $\tau$ on $X$ given by
$$\tau=\{U\mid U\subseteq X_{0}\}\cup\{X\setminus F\mid F\ \mbox{is a finite subset of}\ X_{0}\}.$$ Then we have the following conclusion:

(1) $X$ is a Fr\'echet-Urysohn space.

Let $A\subseteq X$ and $x_{0}\in \overline{A}\setminus A$. Note that

$$\overline{A}=\left\{
             \begin{array}{ll}
              A &\ \ \infty\not\in A\ \mbox{and}\ A\ \mbox{is finite}\, \\
              \{\infty\}\cup A &\ \ \infty\not\in A\ \mbox{and}\ A\ \mbox{is infinite},\\
              A&\ \ \infty\in A.\
             \end{array}
           \right.$$
Then we have that $x_{0}=\infty$ and $A$ is infinite. So we can fix a infinite countable subset $\{a_{1},a_{2},\cdot\cdot\cdot,a_{n},\cdot\cdot\cdot\}\subseteq A$. $\forall\ U\in\tau$ with $\infty\in U$, $U=X\setminus F$ for some finite subset $F\subseteq X_{0}$. Take $m=max\{s\mid a_{s}\in F\}$. One can check that $\forall\ m\leq k$, $a_{k}\in U$. This means that the sequence $\{a_{n}\}_{n\in\mathbb{N}}$ containing in $A$ converges to $\infty$. Therefore, $X$ is a Fr\'echet-Urysohn space and hence $X$ is a sequential space.

(2) $K(X)=\{A\mid A\ \mbox{is a finite subset of}\ \ X_{0}\}\cup\{B\mid \infty\in B\}$.

Since the subspace $X_{0}$ is a discrete space, the compact subsets containing in $X_{0}$ are exactly all finite sets. Let $C\subseteq X$ with $\infty\in C$. As every open cover of $C$ contains an open set $X\setminus F$ for some finite subset $F$, $C$ is compact. So we can conclude that $$K(X)=\{A\mid A\ \mbox{is a finite subset of}\ \ X_{0}\}\cup\{B\mid \infty\in B\}.$$

Set
 $$\mathcal{H}=\{K\in K(X)\mid \infty\in K\}.$$

(3) $\mathcal{H}$ is closed in $P_{S}(X)$.

Actually, $K(X)\setminus\mathcal{H}=\{K\in K(X)\mid\infty\not\in K\}=\Box X_{0}$. Thus, $\mathcal{H}$ is a closed subset in the vietoris topology.

Set $$\mathcal{A}=\{Q\in\mathcal{H}\mid X_{0}\setminus Q\ \mbox{is countable}\}.$$

(4) $\mathcal{A}$ is a sequential closed subset of the closed subspace $\mathcal{H}$.

Let $\{K_{n}\}_{n\in\mathbb{N}}$ be a sequence in $\mathcal{A}$ converging to $K^{\ast}\in\mathcal{H}$. Then for all $n\in\mathbb{N}$, $X_{0}\setminus K_{n}$ is countable. If $X_{0}\setminus K^{\ast}=\emptyset$, then $K^{\ast}=X\in\mathcal{A}$. If $X_{0}\setminus K^{\ast}\neq\emptyset$, choose a $d\in X_{0}\setminus K^{\ast}$. As $\{K_{n}\}_{n\in\mathbb{N}}$ converges to $K^{\ast}$ and $\Box(X\setminus\{d\})\cap\mathcal{H}$ is open, there is a $N$ such that $\{K_{r}\mid N\leq r\}\subseteq\Box(X\setminus\{d\})\cap\mathcal{H}$. Particularly, $d\not\in K_{N}$ and hence $d\in X_{0}\setminus K_{N}$. By the arbitrariness of $d\in X_{0}\setminus K^{\ast}$, we can conclude that $X_{0}\setminus K^{\ast}\subseteq\bigcup\limits_{n\in\mathbb{N}}(X_{0}\setminus K_{n})$ is still countable. Consequently, $X_{0}\setminus K^{\ast}$ is countable and thus $K^{\ast}\in\mathcal{A}$. Showing that $\mathcal{A}$ is sequential closed in the closed subspace $\mathcal{H}$.

(5) $\mathcal{A}$ is not closed subset of the closed subspace $\mathcal{H}$.

$\forall\ U\in\tau$ with $\Box U\cap\mathcal{H}\neq\emptyset$. There is a finite subset $F_{0}\subseteq X_{0}$ such that $$U=(X_{0}\setminus F)\cup\{\infty\}.$$
Choose a $a\in X_{0}\setminus F_{0}$. Then $\mathcal{A}\cap(\Box U)\cap\mathcal{H}\neq\emptyset$ ($\{\infty\}\cup(X_{0}\setminus(F_{0}\cup\{a\}))\in\mathcal{A}\cap(\Box U)\cap\mathcal{H}$). Hence, $\overline{\mathcal{A}}_{\mathcal{H}}=\mathcal{H}$. However, $\mathcal{A}\neq\mathcal{H}$. Consequently, $\mathcal{A}$ is closed in the closed subspace $\mathcal{H}$.

(6) $P_{S}(X)$ is not a sequential space.

By (5), the subspace $\mathcal{H}$ is not a sequential space. Using Lemma 3.4, $P_{S}(X)$ is also not a sequential space.

}
\end{example}

\begin{lemma} {\rm (see \cite{GG03}) Let $X$ be a well-filtered space. Then $K(X)$ is a dcpo and the upper vietoris topology is contained in the Scott topology on $K(X)$.}
\end{lemma}

\begin{lemma}{\rm (see \cite{GG03}) Let $X$ be a $T_{0}$ space and $\{K_{i}\mid i\in I\}\subseteq K(X)$. Then $\bigvee\limits_{i\in I}K_{i}$ exists in $(K(X),\supseteq)$ if and only if
$\bigcap\limits_{i\in I}K_{i}\in K(X)$ and $\bigvee\limits_{i\in I}K_{i}=\bigcap\limits_{i\in I}K_{i}$.}
\end{lemma}

\begin{theorem}{\rm(see \cite{EEE21,EEF30}) Let $X$ be a well-filtered space. If one of the following conditions hold

(1) $X$ is locally compact;

(2) $P_{S}(X)$ is a  Fr\'echet-Urysohn  space,

then the Scott topology agrees with the Vietoris topology $K(X)$. }

\end{theorem}

\begin{corollary}{\rm(see \cite{EEF30}) Let $X$ be a well-filtered first countable topological space. If one of the following conditions hold

(1) $\forall\ K\in K(X)$, $min(K)$ is countable ;

(2) $X$ is metrizable,

then the Scott topology agrees with the Vietoris topology $K(X)$. }

\end{corollary}

\begin{theorem} {\rm Let $X$ be a well-filtered space. If $P_{S}(X)$ is a sequential space, then the upper vietoris topology agrees with the Scott topology on $K(X)$.}
\end{theorem}

\begin{proof} By Lemma 3.6, the upper vietoris topology on $K(X)$ is contained in $\sigma(K(X))$. Let $\mathcal{A}\in\sigma(K(X))$. Assume that there are a $K\in K(X)\setminus\mathcal{A}$ and a sequence $\{K_{n}\}_{n\in\mathbb{N}}\subseteq K(X)\setminus\mathcal{A}$ converging to $K$ in $P_{s}(X)$ such that $K\in\mathcal{A}$. $\forall\ n\in\mathbb{N}$, set $$Q_{n}=K\cup(\bigcup\limits_{n\leq m}K_{m}).$$

$\mathbf{Claim} \ 1$:  Every $Q_{n}$ is a compact saturated set of $X$.

Let $Q_{n}$ be covered by an open sets subfamily $\{U_{i}\mid i\in I\}\subseteq\mathcal{O}(X)$. Then we have $\{U_{i}\mid i\in I\}$ covers the compact set $K$. So there are open sets $U_{i_{1}},U_{i_{1}},\cdot\cdot\cdot,U_{i_{1}}$ such that $K\subseteq U_{i_{1}}\cup U_{i_{1}}\cup\cdot\cdot\cdot\cup U_{i_{1}}$. Take $V=U_{i_{1}}\cup U_{i_{1}}\cup\cdot\cdot\cdot\cup U_{i_{1}}$. As $\{K_{n}\}_{n\in\omega}$ converges to $K$, there is a $n_{0}$ such that $\{K_{s}\mid n_{0}\leq s\}\subseteq\Box V$. For every $0\leq r\leq n_{0}-1$, the compact set $K_{r}$ is also contained in $\bigcup\limits_{i\in I}U_{i}$. This means that $K_{r}\subseteq \bigcup\mathcal{A}_{r}$ for some finite subfamily $\mathcal{A}_{r}\subseteq\{U_{i}\mid i\in I\}$. Now, we can conclude that $$Q_{n}\subseteq \bigcup\limits^{n_{0}-1}\limits_{r=0}(\bigcup\mathcal{A}_{r})\cup(U_{i_{1}}\cup U_{i_{1}}\cup\cdot\cdot\cdot\cup U_{i_{1}}).$$ As a consequence, $Q_{n}$ is compact. Obviously, $Q_{n}$ is a saturated set.

$\mathbf{Claim} \ 2$: $\{Q_{n}\mid n\in\mathbb{N}\}$ is a filtered family and $\bigvee\limits_{n\in\mathbb{N}}Q_{n}=\bigcap\limits_{n\in\mathbb{N}}Q_{n}=K$.

Clearly, $\{Q_{n}\mid n\in\mathbb{N}\}$ is a filtered and $K\subseteq\bigcap\limits_{n\in\mathbb{N}}Q_{n}$. Choose a $q\in\bigcap\limits_{n\in\mathbb{N}}Q_{n}$ and $q\not\in K$. As $K$ is saturated, there exists an open set $U\in\mathcal{O}(X)$ such that $K\subseteq U$ and $q\not\in U$. Using $\{K_{n}\}_{n\in\omega}$ converging to $K$ again, there is a $j\in\mathbb{N}$ such that $\{K_{t}\mid t\in\mathbb{N},j\leq t\}\subseteq\Box U$. It follows that $q\in Q_{j}\subseteq U$, impossible. Thus, $\bigcap\limits_{n\in\mathbb{N}}Q_{n}=K$ and by Lemma 3.7, $\bigvee\limits_{n\in\mathbb{N}}Q_{n}=K$.

Since $\mathcal{A}\in\sigma(K(X))$, by Claim 2, there is a $n_{1}\in\mathbb{N}$ satisfying $Q_{n_{1}}\in\mathcal{A}$. Note that $\mathcal{A}$ is an upper set in $(K(X),\supseteq)$. So we know $K\in\mathcal{A}$, a contradiction. Consequently, $K(X)\setminus\mathcal{A}$ is sequential closed in $P_{s}(X)$. As $P_{s}(X)$ is a sequential space, $K(X)\setminus\mathcal{A}$ is closed in $P_{s}(X)$. Whence, $\mathcal{A}\in\mathcal{O}(P_{s}(X))$. Therefore, the upper vietoris topology and the Scott topology $\sigma(K(X))$ coincide. \end{proof}

\begin{definition}{\rm Let $X$ be a $T_{0}$ space. We say $X$ is a space of countable pseudocharacter if for every $x\in X$, there are countably many open sets $\{U_{n}\mid n\in\mathbb{N}\}$ such that $\up x=\bigcap\limits_{n\in\mathbb{N}}U_{n}$.}
\end{definition}

\begin{remark} {\rm Every first countable $T_{0}$ space is a space of countable pseudocharacter.}

\end{remark}

Sequential spaces and the spaces of countable pseudocharacter are strictly weaker than first countable spaces. Naturally, it arises the following problem:

\vspace*{0.2cm}
$\mathbf{Problem:}$ Let $X$ be a well-filtered space such that $P_{S}(X)$ is a spaces of countable pseudocharacter. Is the Scott topology identical to the vietoris topology on $K(X)$?
\vspace*{0.2cm}

The following Example 3.13 suggests that the above problem has a negative solution.

\begin{example} {\rm Let $\beta\mathbb{N}$ be the Stone-\v{C}ech compactification of the discrete space $\mathbb{N}$ and $p\in\beta\mathbb{N}$ a non-principal ultrafilter of $\mathbb{N}$. So $p\subseteq 2^{\mathbb{N}}$ has the following properties:

(i) $\emptyset\not\in p$;

(ii) if $A\in p, B\in p$, then $A\cap B\in p$;

(iii) if $A\in p$ and $A\subseteq B$, then $B\in p$;

(iv) if $A\subseteq\mathbb{N}$, then $A\in p$ or $\mathbb{N}\setminus A\in p$;

(v) All finite subsets of $\mathbb{N}$ does not belong to $p$;

(vi) $\bigcap\limits_{A\in p}A=\emptyset$.

Let $X=\mathbb{N}\cup\{p\}$. The topology on $X$ is generate by $\mathcal{A}$(as a subbase), where
$$\mathcal{A}=\{\{n\}\mid n\in\mathbb{N}\}\cup\{A\cup\{p\}\mid A\in p\}.$$

Then $X$ satisfies the following conditions.

(1) $X$ is Hausdorff.

$\forall\ x,y\in X$ with $x\neq y$. We just consider the case $x\in\mathbb{N},y=p$. Using (iv) and (v), $\{x\}$ and $(\mathbb{N}\setminus\{x\})\cup\{p\}$ are open neighbourhoods at $x$ and $y$ respectively. Furthermore, $x\not\in(\mathbb{N}\setminus\{x\})\cup\{p\}$. Thus, $X$ is Hausdorff.

(2) $X$ is a space of countable pseudocharacter but not first countable.

$\forall\ x\in \mathbb{N}$, $\up x=\{x\}$ is open. $\forall\ n\in\mathbb{N}$, $\{m\mid n\leq m\in\mathbb{N}\}\in p$. Then $$\{p\}\cup\{m\mid n\leq m\in\mathbb{N}\}$$ is open in $X$. Furthermore, $$\{p\}=\up p=\bigcap\limits_{n\in\mathbb{N}}(\{p\}\cup\{m\mid n\leq m\in\mathbb{N}\}).$$ Thus, $X$ is a space of countable pseudocharacter.

$X$ does not exist a countable neighbourhood base at $p$. Suppose $\{D_{n}\mid n\in\mathbb{N}\}\subseteq\mathcal{A}$ is a neighbourhood subbase at $p$. Then $$\{D_{0}\cap D_{1}\cap\cdot\cdot\cdot\cap D_{p}\mid p\in\mathbb{N}\}$$ is a neighbourhood base at $p$. In convenience, we set $\forall\ n\in\mathbb{N}$, $D_{n}=\{p\}\cup F_{n}$ and $F_{n}\in p$. Take $G_{n}=F_{0}\cap F_{1}\cap\cdot\cdot\cdot\cap F_{n}$, for all $n\in\mathbb{N}$. Then by (ii), $G_{n}\in p$ and by (iv)(v), every $G_{n}$ is infinite. So for all $n\in\mathbb{N}$, we can choose a $x_{n}\in G_{n}$ and a $y_{n}\in G_{n}$ such that every pair elements coming form $\{x_{n}\mid n\in\mathbb{N}\}\cup\{y_{n}\mid n\in\mathbb{N}\}$ are different. Set
$$I=\{x_{n}\mid n\in\mathbb{N}\},J=\{y_{n}\mid n\in\mathbb{N}\}.$$ As $I\cap J=\emptyset$, one of $I$ and $J$ does not belong to $p$. Suppose $I\not\in p$. By (iv) $\mathbb{N}\setminus I\in p$ and thus
$$H=\{p\}\cup(\mathbb{N}\setminus I)$$ is open in $X$.

For every $k\in\mathbb{N}$, $$x_{k}\in I\subseteq G_{n}\subseteq D_{0}\cap\cdot\cdot\cdot\cap D_{K},$$ but $x_{k}\not\in H$. It follows that for all $k\in \mathbb{N}$, $D_{1}\cap\cdot\cdot\cdot\cap D_{k}\not\subseteq H$. Therefore, $X$ is not first-countable.

(3) $K(X)=\{F\mid F\ \mbox{is a finite subset of}\ X\}$.

Let $K$ be a infinite subset of $X$. Then $K\cap\mathbb{N}$ is finite. Split $K\cap\mathbb{N}$ into two infinite disjoint subsets $B$ and $C$. Then by (i)(ii), one of $B$ and $C$ does not belong to $p$. Without loss of generality, we let $B\not\in p$. By (iv), $\mathbb{N}\setminus B\in p$. Let $$U_{B}=(\mathbb{N}\setminus B)\cup\{p\}.$$ Consider the open set family
$$\mathcal{B}=\{U_{B}\}\cup\{\{b\}\mid b\in B\}.$$ Clearly, $\mathcal{B}$ covers $K$ but $\mathcal{B}$ has no finite subfamily covering $K$. This implies that $K$ is not compact. Hence, the compact subsets of $X$ are all finite subsets.

(4) $P_{S}(X)$ is a space of countable pseudocharacter.

Let $Q\in K(X)$. We consider the following two cases.

(4-1) $p\not\in Q$.

Now, $Q\subseteq\mathbb{N}$ is open. Then $\up_{P_{s}(X)}Q=\Box Q$.

(4-2) $p\in Q$.

By (3), $Q=\{p\}\cup F$ for some finite subset $F\subseteq\mathbb{N}$. $\forall\ r\in\mathbb{N}$, set $$U_{r}=\{p\}\cup F\cup\{t\mid r\leq t, t\in\mathbb{N}\}.$$ By (2), $U_{r}$ is open in $X$. Besides,
\begin{center}$\begin{array}{lll}
\bigcap\limits_{r\in\mathbb{N}}\Box U_{r}&=&\{R\in K(X)\mid \forall\ r\in\mathbb{N}, R\subseteq U_{r}\}\\
&=&\{R\in K(X)\mid  R\subseteq \bigcap\limits_{r\in\mathbb{N}}U_{r}\}\\
&=&\{R\in K(X)\mid  R\subseteq Q\}\\
&=&\up_{P_{S}(X)}Q.\\
\end{array}$\end{center}
Therefore, $P_{S}(X)$ is a space of countable pseudocharacter.

(5) The Scott topology and the vieotris topology on $K(X)$ do not coincide.

Let $\mathcal{C}=\{\{p\}\}$. Then $\mathcal{C}$ is an upper set in $K(X)$. By (3), every element of $(K(X),\supseteq)$ is a compact element. So $\mathcal{C}$ is a Scott open set in $K(X)$. Suppose there is an open set $W\in\mathcal{O}(X)$ such that $\{p\}\in \Box U\subseteq\mathcal{C}$. Then there is a $E\in p$ such that $E\cup\{p\}\subseteq U$. By (v), $E$ is a infinite subset of $\mathbb{N}$. $\Box U$ is finite because all finite subsets of $E$ belong to $\Box U$. This contradicts with $\Box U\subseteq\mathcal{C}$. So $\mathcal{C}$ is not open in the vietoris topology.
}
\end{example}

In \cite{EEF30}, Xu and Yang proposed the following problem.

\vspace*{0.2cm}
$\mathbf{Problem:}$ Let $X$ be a first countable well-filtered space. Does the Scott topology and the Vietoris topology on $K(X)$ coincide?
\vspace*{0.2cm}

In the following, we present  a first countable Hausdorff space $X$ which violates all known sufficient conditions (Theorem 3.8, Corollary 3.9 and Theorem 3.10) for the coincidence of the Scott topology and the Vietoris topology on $K(X)$.

\begin{example} {\rm Let $C=\{(x,y)\mid x^{2}+y^{2}=1\}$ be the unit circle endowed with the usual subspace topology of the Euclidean space $\mathbb{R}^{2}$. Set $A(C)=C\times\{0,1\}$. The topology on $A(C)$ is defined by:

(i) $\forall\ (x,0)\in C\times\{0\}$, $(x,0)$ has the neighbourhood system
$$\{(U\times\{0,1\})\setminus(x,1)\mid U\ \mbox{is the open neighbourhood at}\ x \ \mbox{in}\ C\};$$

(ii) $\forall\ (x,1)\in C\times\{1\}$, $(x,1)$ is a isolated points.

In convenience, we let $N(x,U)=(U\times\{0,1\})\setminus(x,1)$

Let $X=A(C)\times\mathbb{Q}$ be the product space of $A(C)$ and $\mathbb{Q}$, where $\mathbb{Q}\subseteq\mathbb{R}$ is the rational numbers set endowed with the subspace topology. The we have the following statements:

(1) $X$ is Hausdorff.

$\forall\ (a,i),(b,j)\in A(C)$ with $(a,i)\neq(b,j)$, we consider the following cases:

(1-1) $i=j=1$. Then $\{(a,1)\}$ and $\{(b,1)\}$ are open sets separating $a,1)$ and $(b,1)$.

(1-2) $i=j=0$. We can choose two open sets $U_{1},V_{1}\in\mathcal{O}(C)$ with $a\in U_{1}$ $b\in V_{1}$ such that $U_{1}\cap V_{1}=\emptyset$. Now $N(a,U_{1})$ and $N(b,V_{1})$ are open neighbourhood at $a,0)$ and $(b,0)$ respectively. Furthermore, we have $$N(a,U_{1})\cap N(b,V_{1})=\emptyset.$$

(1-3) $i=0,j=1$. Select a open set $U_{2}\in\mathcal{O}(C)$ with $a\in U_{2}$, then $N(a,U_{2})$ and $\{(b,1)\}$ are disjoint open neighbourhood at $(a,0)$ and $(b,1)$ respectively.

(1-4) $i=1,j=0$. It is similar to the case (1-3).

By cases (1-1)-(1-4),$A(C)$ is Hausdorff. Note that $\mathbb{Q}$ is Hausdorff. So $X$ is also Hausdorff.

(2) $X$ is first countable.

Since $C$ is first countable, $A(C)$ is first-countable. Clearly, $\mathbb{Q}$ is first countable. Hence, $X$ is first-countable.

(3) $X$ is not metrizable.

Let $\mathcal{U}$ be an open cover of $A(C)$. Then $\forall\ x\in C$, there are an open set $U_{x}\in\mathcal{O}(C)$ and a $G_{x}\in\mathcal{U}$ such that $(x,0)\in N(x,U_{x})\subseteq G_{x}$. Obviously, $\{U_{x}\mid x\in C\}$ is an open cover of $C$. Since $C$ is compact, there are $x_{1},x_{2},\cdot\cdot\cdot,x_{n}\in C$ such that $C=U_{x_{1}}\cup U_{x_{2}}\cup\cdot\cdot\cdot\cup U_{x_{n}}$. Then we know
$$(C\times\{0\})\cup((C\setminus\{x_{1},x_{2},\cdot\cdot\cdot,x_{n}\})\times\{1\})\subseteq \bigcup\limits_{1\leq k\leq n}N(x_{k},U_{x_{k}})\subseteq \bigcup\limits_{1\leq k\leq n}G_{x_{k}}.$$ For every $1\leq k\leq n$, choose a $H_{x_{k}}\in\mathcal{U}$ with $(x_{k},1)\in H_{x_{k}}$.
It follows that $$\{G_{x_{k}})\mid 1\leq k\leq n\}\cup\{H_{x_{k}}\mid 1\leq r\leq n\}$$ is a finite subcover of $A(C)$. Thus, $A(C)$ is compact.

Since $\{\{(x,1)\}\mid x\in C\}$ is a uncountable family consisting of disjoint open sets $A(C)$, $A(C)$ is not second countable. Every compact metrizable space is second countable(see \cite{RE77}). This means that $A(C)$ is not metrizable. Therefore, $X=A(C)\times\mathbb{Q}$ is not metrizable.

(4) $X$ contains a uncountable compact subset.

Consider the subspace $K_{0}=(C\times\{0\})\times\{0\}$ of $X$. One can check that the subspace $K_{0}$ is homeomorphic to the unit circle $C$. So $K_{0}$ is a uncountable compact subset of $X$.

(5) $X$ is not compact and not locally compact.

As $\mathbb{Q}$ is not compact, $X=A(C)\times\mathbb{Q}$ is not compact.

Suppose $X=A(C)\times\mathbb{Q}$ is locally compact. Fix $a\in A(C)$ and $q\in\mathbb{Q}.$
By local compactness, there exist an open set $V\subseteq X$ and a compact set $K\subseteq X$ such that $(a,q)\in V\subseteq K\subseteq X.$ Then $\pi_2(V)$ is an open subset of $\mathbb{Q}$ containing $q$ and $\pi_2(K)$ is compact in $\mathbb{Q}.$
Therefore, $q\in\pi_2(V)\subseteq\pi_2(K)$. Thus every point of $\mathbb{Q}$ has an open neighborhood contained in a compact subset of $\mathbb{Q},$ which would imply that $\mathbb{Q}$
is locally compact.  We now show that this is impossible. Fix $q\in\mathbb{Q},$ and suppose that there exist an open set $W\subseteq\mathbb{Q}$ and a compact set $L\subseteq\mathbb{Q}$
such that $q\in W\subseteq L$. Since $W$ is an open neighborhood of $q$ in $\mathbb{Q}$,
there exists $\varepsilon>0$ such that $(q-\varepsilon,q+\varepsilon)\cap\mathbb{Q}\subseteq W\subseteq L$. Choose an irrational number $\alpha\in(q-\varepsilon,q+\varepsilon)$
and a sequence $\{q_n\}_{n\in\mathbb{N}}$ of rational numbers satisfying $q_n\in(q-\varepsilon,q+\varepsilon)\cap\mathbb{Q}$ and converging to $\alpha$. Then $q_n\in L$ for every $n$. $L$ is compact in $\mathbb{R}$ and therefore $L$ is closed in $\mathbb{R}$.
Since $\{q_n\}_{n\in\mathbb{N}}$ converges to $\alpha$ in $\mathbb{R},$
it follows that $\alpha\in L\subseteq\mathbb{Q}$, contradicting the fact that $\alpha\not\in\mathbb{Q}$. Therefore, $\mathbb{Q}$ is not locally compact. Consequently, $X=A(C)\times\mathbb{Q}$ is not locally compact.

(6) $P_{S}(X)$ is not a sequential space.

Set $C_{0}=C\times\{0\}$ and $D=C\times\{1\}$. For every countable subset $E\subseteq D$, take
$$M_{E}=(C_{0}\cup(D\setminus E))\times\{0\}=(A(C)\setminus E)\times\{0\}$$.
Since $E\subseteq D$ is open on $A(C)$, the closed subset $A(C)\setminus E$ is compact in $A(C)$. Hence, $M_{E}$ is a compact subset in $X$.$$\mathcal{B}=\{H\in K(X)\mid M_{E}\subseteq H,\ \mbox{where}\ E\subseteq D\ \mbox{is a countable subset}\}.$$

$\mathbf{Claim} \ 1$: $\mathcal{B}$ is not closed in $P_{S}(X)$.

Take $K_{0}=C_{0}\times\{0\}$. By (4), $K_{0}\in K(X)$. However, $K_{0}\not\in\mathcal{B}$ because $$K_{0}\cap (D\times\{0\})=\emptyset.$$
Let $W\in\mathcal{O}(X)$ with $K_{0}\in\Box W$. Define a mapping $f:A(C)\longrightarrow X$ by
$$\forall\ a\in A(C), f(a)=(a,0).$$
Then $f$ is continuous. So $$f^{-1}(U)=\{b\in A(C)\mid (b,0)\in U\}$$ is open in $A(C)$. Consequently, the closed subset $A(C)\setminus f^{-1}(U)$ is compact in $A(C)$. Furthermore, $A(C)\setminus f^{-1}(U)\subseteq D$ since $C_{0}=f^{-1}(K_{0})\subseteq f^{-1}(U)$. Then $A(C)\setminus f^{-1}(U)$ is compact in discrete subspace $D$. As a consequence, $A(C)\setminus f^{-1}(U)$ is finite subset of $D$. Take $$F=A(C)\setminus f^{-1}(U).$$ Then we have
$$M_{F}=(C_{0}\cup(D\setminus F))\times\{0\}=f^{-1}(U)\times\{0\}\subseteq U.$$
Whence, $M_{F}\in\mathcal{B}\cap(\Box U)$. It follows that $K_{0}\in\overline{\mathcal{B}}_{P_{S}(X)}$.

$\mathbf{Claim} \ 2$: $\mathcal{B}$ is sequential closed in $P_{S}(X)$.

Let $\{Q_{n}\}_{n\in\mathbb{N}}$ be a sequence in $\mathcal{B}$ converging to $Q$. Then $\forall\ n\in\mathbb{N}$, there exists a countable subset $E_{n}\subseteq D$ such that $M_{E_{n}}\subseteq Q_{n}$. Set $$E^{\ast}=\bigcup\limits_{n\in\mathbb{N}}E_{n}.$$ Clearly, $E^{\ast}$ is still countable. $\forall\ e\in D\setminus E^{\ast}$, we have that for all $n\in\mathbb{N}$, $(e,0)\in M_{E_{n}}\subseteq Q_{n}$. Assume $(e,0)\not\in Q$. As $X$ is Hausdorff, $U_{e}=X\setminus\{(e,0)\}$ is open and $Q\in\Box U_{e}$. As $\{Q_{n}\}_{n\in\mathbb{N}}$ converges to $Q$, there is a $m_{0}\in\mathbb{N}$ such that $\{Q_{s}\mid m_{0}\leq s\}\subseteq\Box U_{e}$. Whence, $Q_{m_{0}}\subseteq U_{e}$. This contradicts with $(e,0)\in Q_{m_{0}}$. It follows that $(D\setminus E^{\ast})\times\{0\}\subseteq Q$. $\forall\ (z,0)\in C_{0}$, $\forall W\in\mathcal{O}(C)$ with $z\in W$. The set $N(x,W)\cap (D\setminus E^{\ast})$ is nonempty because the uncountable set $(W\times\{1\})\setminus\{(z,1)\}$ can not be contained in the countable subset $E^{\ast}$. This implies that $C_{0}\subseteq\overline{D\setminus E^{\ast}}_{A(C)}$. In addition, the compact subset $Q$ is closed because $X$ is Hausdorff. Consequently, $$C_{0}\times\{0\}\subseteq\overline{(D\setminus E^{\ast})}_{A(C)}\times\overline{\{0\}}_{\mathbb{Q}}=\overline{(D\setminus E^{\ast})\times\{0\}}_{X}\subseteq \overline{Q}_{X}=Q.$$ This implies that
$$M_{E^{\ast}}=(C_{0}\cup(D\setminus E^{\ast}))\times\{0\}=(C_{0}\times\{0\})\cup((D\setminus E^{\ast})\times\{0\})\subseteq Q.$$ Equivalently, $Q\in\mathcal{B}$. Thus, $\mathcal{B}$ is sequential closed in $P_{S}(X)$.

By two claims, $P_{S}(X)$ is not a sequential space.

(7) The Scott topology $\sigma(K(X))$ contains the upper vietoris topology on $K(X)$.

By Lemma 3.6, the statement (7) holds.

}
\end{example}

\begin{problem} {\rm Let $X$ be the topological space showed in Example 3.7. Does the Scott topology agree with the vietoris topology on $K(X)$?}

\end{problem}

\section{Applications of sequential spaces to sober spaces}\label{sec:fm}

In this section, we present a sufficient condition for a well-filtered space to be sober.

\begin{theorem} {\rm Let $X$ be a $\omega$-well-filtered coherent $d$-space. If $X\times X$ is a sequential space, then $X$ is sober.}
\end{theorem}

\begin{proof} Let $A\subseteq X$ be an irreducible subset. Set
$$R_{A}=\{(x,y)\in \overline{A}\times\overline{A}\mid \exists z\in\overline{A}\times\overline{A}, x\leq z,y\leq z\}.$$

$\mathbf{Claim}\ 1$: $R_{A}$ is sequential closed in $X\times X$.

Suppose $\{(a_{n},b_{n})\}_{n\in\mathbb{N}}$ coming form $R_{A}$ converges to $(a,b)$. Then for every $n\in\mathbb{N}$, there is $c_{n}\in\overline{A}$ such that $a_{n}\leq c_{n},b_{n}\leq c_{n}$. Let $U,V\in\mathcal{O}(X)$ with $(a,b)\in U\times V$. Since $\{(a_{n},b_{n})\}_{n\in\mathbb{N}}$ converges to $(a,b)$, there exists a $n_{0}$ such that $$\{(a_{m},b_{m})\mid m\in\mathbb{N},n_{0}\leq m\}\subseteq U\times V.$$ Since $U,V$ are upper sets in $X$, $$\{(c_{m},c_{m})\mid m\in\mathbb{N},n_{0}\leq m\}\subseteq U\times V.$$ This implies that the sequence $\{(c_{n},c_{n})\}$ converges to $(a,b)$ again. Consequently, the sequence $\{c_{n}\}_{n\in\mathbb{N}}$ converges to $x$ and $y$ in $X$. For every $k\in\mathbb{N}$, take
$$P_{k}=\up\{c_{m}\mid m\in\mathbb{N},k\leq m\}\cup\up a\ \mbox{and}\ Q_{k}=\up\{c_{m}\mid m\in\mathbb{N},k\leq m\}\cup\up b.$$
As $\{c_{n}\}_{n\in\mathbb{N}}$ converges to $a$ and $b$, we can check that $P_{k}$ and $Q_{k}$ are compact saturated set in $X$ for all $k\in\mathbb{N}$. $P_{k}\cap Q_{k}$ is still a compact saturated set in $X$ because $X$ is coherent. Assume that $$\overline{A}\cap(\bigcap\limits_{n\in\mathbb{N}}(P_{k}\cap Q_{k}))=\emptyset.$$ Equivalently, $$\bigcap\limits_{n\in\mathbb{N}}(P_{k}\cap Q_{k})\subseteq X\setminus \overline{A}.$$ As $X$ is $\omega$-well-filtered and $\{P_{k}\cap Q_{k}\mid k\in\mathbb{N}\}$ is a filtered family of compact saturated sets, there is a $k_{0}$ such that $$P_{k_{0}}\cap Q_{k_{0}}\subseteq X\setminus \overline{A}.$$ Whence $c_{k_{0}}\in P_{k_{0}}\cap Q_{k_{0}}\subseteq X\setminus \overline{A}$, a contradiction. As a consequence, $$\overline{A}\cap(\bigcap\limits_{n\in\mathbb{N}}(P_{k}\cap Q_{k}))\neq\emptyset.$$ Choose a $$z\in\overline{A}\cap(\bigcap\limits_{n\in\mathbb{N}}(P_{k}\cap Q_{k}))\neq\emptyset.$$ Assume that $a\not\leq z$. Then $X\setminus \dn z$ is an open neighbourhood at $a$. It follows form $\{c_{n}\}_{n\in\mathbb{N}}$ converging to $a$ that there is $n_{1}$ such that $\{c_{n}\mid n\in\mathbb{N},n_{1}\leq n\}\subseteq X\setminus \dn z$. Since $z\in\bigcap\limits_{k\in\mathbb{N}}P_{k}$ and $z\not\in\up a$, $z\in\bigcap\limits_{k\in\mathbb{N}}(\up\{c_{m}\mid m\in\mathbb{N},k\leq m\})$. So for every $k\in\mathbb{N}$ there exists a $r_{k}\in\mathbb{N}$ with $k\leq r_{k}$ satisfying $c_{r_{k}}\leq z$. In particular, $c_{r_{n_{1}}}\leq z$ ($n_{1}\leq r_{n_{1}}$), a contradiction. So we conclude that $a\leq z$. Similarly, $b\leq z$. Therefore, $z\in R_{A}$ and $R_{A}$ is sequential closed in $X\times X$.

Since $X\times X$ is a sequential space, $R_{A}$ is closed in $X\times X$. Set
$$\bigtriangleup A=\{(x,x)\mid x\in A\}.$$ Obviously, $\bigtriangleup A\subseteq R_{A}$.

$\mathbf{Claim}\ 2$: $\overline{A}$ is directed.

$\forall\ p,q\in \overline{A}$ and $\forall\ E,F\in\mathcal{O}(X)$ with $(p,q)\in E\times F$. As $\overline{A}$ is irreducible, $\overline{A}\cap E\cap F\neq\emptyset$. Select a $s\in \overline{A}\cap E\cap F$. Then $(s,s)\in \bigtriangleup A\cap (E\times F)$. This implies that $$(p,q)\in\overline{\bigtriangleup A}_{X\times X}.$$  Since $R_{A}$ is closed and $\bigtriangleup A\subseteq R_{A}$, $(p,q)\in\overline{\bigtriangleup A}_{X\times X}\subseteq R_{A}$. Consequently, $p,q$ has an upper bound in $\overline{A}$. Thus, $\overline{A}$ is directed.

As $X$ is a $d$-space, $\bigvee \overline{A}$ exists and $\overline{A}=\overline{\{\bigvee \overline{A}\}}$. Therefore, $X$ is sober.
\end{proof}

\begin{corollary} {\rm Let $X$ be a well-filtered coherent space. If $X\times X$ is a sequential space, then $X$ is sober.}
\end{corollary}

\begin{lemma} {\rm (See \cite{XLOP17}) Let $L$ be a complete lattice. Then $\Sigma L$ is a well-filtered coherent space.}
\end{lemma}

\begin{corollary}{\rm Let $L$ be a complete lattice. If $\Sigma L\times \Sigma L$ is a sequential space, then $\Sigma L$ is sober.}
\end{corollary}

\begin{lemma}{\rm(See \cite{GG03}) Let $L$ be a complete lattice. If $\Sigma L\times\Sigma L=\Sigma(L\times L)$, then $\Sigma L$ is sober.}
\end{lemma}

\begin{theorem} {\rm Let $P$ and $Q$ be two posets. If $\Sigma P\times \Sigma Q$ is a sequential  space, then $\Sigma (P\times Q)=\Sigma P \times \Sigma Q$.}
\end{theorem}

\begin{proof}
Set $$\tau=\sigma(P)\times\sigma(Q)
    \quad\text{and}\quad
    \rho=\sigma(P\times Q).$$
It is always the case that $$\sigma(P)\times\sigma(Q)\subseteq\sigma(P\times Q).$$
Conversely, let $A\subseteq P\times Q$ be
$\rho$-closed. Since $(P\times Q,\tau)=\Sigma P\times\Sigma Q$
is sequential, it suffices to show that $A$ is sequentially closed
with respect to $\tau$. Let $\{(p_n,q_n)\}_{n\in\mathbb{N}}$ be a sequence in $A$ and suppose that
the sequence $\{(p_n,q_n)\}_{n\in\mathbb{N}}$ converges to $(p,q)$ in $\Sigma P\times\Sigma Q$. Then $\{p_n\}_{n\in\mathbb{N}}$ ($\{q_{n}\}_{n\in\mathbb{N}}$) converges to $p$($q$) in $\Sigma P$ ($\Sigma Q$). Assume that $(p,q)\notin A$. Put $$T=(P\times Q)\setminus A.$$ Then $T\in\sigma(P\times Q)$ and $(p,q)\in T$.

Define a mapping $f:P\longrightarrow\mathcal O(\Sigma Q)$ by
 $$ \forall\ a\in P, f(a)=\{b\in Q:(a,b)\in T\}.$$
For every $a\in P$, we can check that $f(a)$ is Scott open in $Q$. Furthermore,
$f$ is Scott continuous. Since $(p,q)\in T$, we have $q\in f(p)$. As
$\{q_{n}\}_{n\in\mathbb{N}}$) converges to $q$, there exists $n_{0}$ such
that for every $n_0\leq n$, $q_n\in f(p)$. Let
$$E=\{q\}\cup\{q_n\mid n_0\leq n\}.$$
Then $E$ is compact in $\Sigma Q$, since it consists of a
convergent sequence together with its limit, and $E\subseteq f(p)$. Define
$$\Box E=\{G\in\mathcal O(\Sigma Q)\mid E\subseteq G\}.$$
Then $\Box E$ is Scott open in $\mathcal O(\Sigma Q)$. Since $E\subseteq f(p)$, we have $f(p)\in\Box E$. By the Scott
continuity of $f$, $W=f^{-1}(\Box E)$
is a Scott open neighbourhood of $p$ in $P$. Since $\{p_n\}_{n\in\mathbb{N}}$ converges to $p$,
there exists $n_0\leq n_1$ such that $p_{n}\in W$ for all $n_1\leq n$. In particular, $p_{n_1}\in W$, so $E\subseteq f(p_{n_1})$. Since $q_{n_1}\in E$, it follows that $q_{n_1}\in f(p_{n_1})$ and therefore $p_{n_1},q_{n_1})\in T$. This contradicts $(p_{n_1},q_{n_1})\in A$. Hence $(p,q)\in A$, and therefore $A$ is sequentially closed in
$(P\times Q,\tau)$. Since $\Sigma P\times \Sigma Q$  is sequential, $A$ is
$\tau$-closed. Thus every $\rho$-closed set is $\tau$-closed,
which gives $\rho\subseteq\tau$. Therefore
$\sigma(P\times Q)=\sigma(P)\times\sigma(Q)$. \end{proof}

Coincidentally, by Lemma 4.5 and Theorem 4.6, we can also obtain Corollary 4.4.


\begin{proposition} {\rm Every countable $T_{0}$ space $X$ is a space of countable pseudocharacter.}
\end{proposition}

\begin{proof} Let $x\in X$. Then $X\setminus \up x$ is countable and set
$$X\setminus \up x=\{a_{0},a_{1},\cdot\cdot\cdot,a_{n},\cdot\cdot\cdot\}.$$
For every $n\in\mathbb{N}$, since $a_{n}\not\in\up x$, there is an open set $U_{n}$ such that $x\in U_{n}$ and $a_{n}\not\in U_{n}$. Then we can see
$$\bigcap\limits_{n\in\mathbb{N}}U_{n}=\up x.$$ Thus $X$ is a space of countable pseudocharacter.
\end{proof}

As we all known, sequential spaces and the spaces of countable pseudocharacter are strictly weaker than first countable spaces. Naturally, it arises the following problem:

\vspace*{0.2cm}
$\mathbf{Problem:}$ Let $X$ be a well-filtered coherent space such that $X\times X$ is a spaces of countable pseudocharacter. Is $X$ sober?
\vspace*{0.2cm}

The following example 4.8 gives a negative solution for the above problem.

\begin{example} {\rm Let $L$ be the countable complete lattice constructed by Miao, Xi, Li and Zhao (\cite{AB75}). Then by Lemma 4.3, $X=\Sigma L$ is a coherent well-filtered but not sober space. By Proposition 4.7, $X\times X$ is not a space of countable pseudocharacter.}
\end{example}

In general, the sequentiality of $X\times X$ is not necessary for the sobriety of $X$. See the following Example 4.9.

\begin{example} {\rm Let $X=[0,\omega_{1}]$, where $\omega_{1}$ is the least uncountable ordinal number. The topology on $X$ has a subbase
$$\mathcal{A}=\{(\alpha,\beta)\mid \alpha,\beta\in X, \alpha<\beta\}\cup\{[0,\alpha)\mid 0<\alpha\}\cup\{(\beta,\omega_{1}]\mid \beta<\omega_{1}\}.$$
Then the following conclusions hold.

(1) $X$ is Hausdorff.

It is not difficult to check.

By(1), $X$ is a coherent sober space.

(2) $X\times X$ is not a sequential space.

Let $A=[0,\omega_{1})\times\{0\}$.

$\mathbf{Claim}\ 1$: $A$ is not closed in $X\times X$.

Take $$p=(\omega_{1},0)$$ and $U,V\in\mathcal{O}(X)$ with $\omega_{1}\in U,0\in V$. Then there is a $\gamma<\omega_{1}$ and $0<\delta$ such that $$(\gamma,\omega_{1}]\subseteq U\ \mbox{and }\ [\delta,0)\subseteq V.$$ As $\gamma<\omega_{1}$, there is a $\mu\in(\gamma,\omega_{1})$, Now, we have $(\mu,0)\in((\gamma,\omega_{1}]\times[\delta,0))\cap A\subseteq (U\times V)\cap A$. This means that $p\in \overline{A}$. $A$ is closed because $p\in \overline{A}\setminus A$.

$\mathbf{Claim}\ 2$: $A$ is sequential closed in $X\times X$.

Suppose $\{(\xi_{n},0)\}_{n\in\mathbb{N}}$ is a consequence in $A$ converging to $(\theta,\vartheta)$. Then $\{\xi_{n}\}_{n\in\mathbb{N}}$ converges to $\theta$ and the constant sequence 0 converges to $\vartheta$. As $X$ is Hausdorff, $\vartheta=0$. Since all ordinal numbers $\alpha_{n}$ is countable, $$\varphi={\rm sup\{\xi_{n}\mid n\in\mathbb{N}\}}=\bigcup\{\xi_{n}\mid n\in\mathbb{N}\}$$ is a countable ordinal number. Assume that $\theta=\omega_{1}$. Consider the open neighbourhood $(\varphi,\omega_{1}]$ of $\theta$. Since $\{\xi_{n}\}_{n\in\mathbb{N}}$ converges to $\theta$, there is a $m_{0}$ such that $\{\xi_{k}\mid m_{0}\leq k\}\subseteq(\varphi,\omega_{1}]$. Particularly, $\varphi<\xi_{m_{0}}$, a contradiction. As a consequence, $\theta<\omega_{1}$. In other words, $(\theta,\vartheta)\in A$. Hence, $A$ is sequential closed.

By two claims, $X\times X$ is not a sequential space.}
\end{example}

\begin{proposition} {\rm Let $Y$ be a sequential space. If $X$ is the retract of $Y$, then $X$ is a sequential space.}
\end{proposition}

\begin{proof} Let $f:X\longrightarrow Y$ and $g:Y\longrightarrow X$ be continuous mappings such that $g\circ f=id_{X}$. Suppose $A\subseteq X$ is a sequential closed subset.

$\mathbf{Claim}$: $g^{-1}(A)$ is a closed subset in $Y$.

Let $\{x_{n}\}_{n\in\mathbb{N}}$ be a sequence in $g^{-1}(A)$ converging to $y\in Y$. Since $f$ is continuous, $\{g(x_{n})\}_{n\in\mathbb{N}}$ converges to $g(y)$. As $A$ is a sequential closed and every $g(x_{n})\in A$, $g(y)\in A$. It follows that $g^{-1}(A)$ is sequential closed in $Y$. Consequently, $g^{-1}(A)$ is closed in $Y$ because $Y$ is a sequential space.

By the continuity of $f$, $A=f^{-1}(g^{-1}(A))$ is closed in $X$. Therefore, $X$ is a sequential space.
\end{proof}

Note that $X$ is a retract of $X\times X$. By Proposition 4.10, $X$ is a sequential space if $X\times X$ is a sequential space. Based on theorem 4.1, it arises the following problem.

\vspace*{0.2cm}
$\mathbf{Problem:}$ Is every well-filtered sequential space always sober?
\vspace*{0.2cm}

Unfortunately, the following Example 4.12 shows that the solution is negative. Firstly, the following Lemma 4.11 is necessary.

\begin{lemma} {\rm Let $P$ be a dcpo. If for every directed subset $D\subseteq P$, there is a increasing sequence $\{d_{n}\}_{n\in\mathbb{N}}$ contained in $D$ such that $$\bigvee D=\bigvee\limits_{n\in\mathbb{N}}d_{n},$$ then $\Sigma P$ is a sequential space.}
\end{lemma}

\begin{proof} Let $A\subseteq P$ be a sequential closed subset of $\Sigma P$.

$\mathbf{Claim\ 1}$: $A$ is an lower set.

Suppose $x\in A$ and $y\leq x$. Consider the constant sequence $\{x_n\}_{n\in\mathbb{N}}$ with value $x$. Since every Scott open set is an upper set, $\{x_n\}_{n\in\mathbb{N}}$ converges to $y$. As $A$ is sequential closed, $y\in A$.

$\mathbf{Claim\ 2}$: $A$ is closed under directed suprema.

Fix a $D\subseteq A$ being a directed subset, then there is a increasing sequence $\{d_{n}\}_{n\in\mathbb{N}}$ in $D$ such that $\bigvee D=\bigvee\limits_{n\in\mathbb{N}}d_{n}$. Take a $U\in\sigma(P)$ with $\bigvee D\in U$. By the definition of Scott open sets, $$U\cap\{d_{n}\mid n\in\mathbb{N}\}\neq\emptyset.$$ Choose a $d_{k_{0}}\in U$. Then $\{d_{n}\mid k_{0}\leq n\}\subseteq U$ as $\{d_{n}\}_{n\in\mathbb{N}}$ is increasing and $U$ is upper. It follows that $\{d_{n}\}_{n\in\mathbb{N}}$ converges to $\bigvee D$. Using the sequential closedness of $A$ again, $\bigvee D\in A$.

By Claim 1 and Claim 2, $A$ is Scott closed. Therefore, $\Sigma P$ is a sequential space.\end{proof}

\begin{example} {\rm Let $P=X\cup P_{0}$, where $X=(0,1]$ and $$P_{0}=\{(k,a,b)\in\mathbb{R}^{3}\mid 0<k<1,0<b\leq a\leq 1\}.$$
The order $\leq_{P}$ on $P$ is defined by:

(i) $\forall\ x_{1},x_{2}\in X$, $x_{1}\leq_{P}x_{1}$ iff $x_{1}=x_{2}$;

(ii) $\forall\ (k_{1},a_{1},b_{1}),(k_{2},a_{2},b_{2})\in P_{0}$,
 $$(k_{1},a_{1},b_{1})\leq_{P}(k_{2},a_{2},b_{2})\ \mbox{iff}\ \ k_{1}\leq k_{2}, a_{1}=a_{2},b_{1}=b_{2};$$

(iii) $\forall\ (k,a,b)\in P_{0}$ and $x\in X$, $(k,a,b)\leq_{P} x$ iff $a=x$ or $kb\leq x< b$.

The $(P,\leq_{P})$ is a dcpo whose Scott topology is well-filtered but not sober(\cite{ME182}). The maximal elements of $P$ are precisely all $x\in X$. Let $D\subseteq P$ be a directed subset. There are only 2 cases:

Case 1: $D\cap X\neq\emptyset$.

Then $D$ has a maximal element in $X$.

Case 2: $D\subseteq P$.

Take $C_{a,b}=\{(k,a,b)\mid k\in(0,1)\}$. Then we can deduce that $D\subseteq C_{a,b}$. In convenience, write $D$ as $\{(k,a,b)\mid k\in K\}$ for some $K\subseteq(0,1)$. Furthermore,

$$\bigvee D=\left\{
             \begin{array}{ll}
              (k,a,b), &\ \ k{\rm=sup}\ K<1, \\
              a, &\ \ {\rm sup}\ K=1.
             \end{array}
           \right.$$
One can check that every directed subset $D\subseteq P$ contains a increasing sequence whose supermum is $\bigvee D$. By Lemma 4.11, $\Sigma P$ is a sequential space.
}
\end{example}

Finally, we present a conclusion that every sequential space $X$ builds a connection for the sobriety between $\sigma(K(X))$ and $O(X)$ with respected to the Scott topology.

\begin{lemma}{\rm (see \cite{GG03}) Let $X$ be a  $T_{0}$ space and $Y$ a sober space. If $X$ is a retract of $Y$, then $Y$ is sober.}
\end{lemma}

\begin{theorem} {\rm Let $X$ be a sequential $T_{0}$ space. If $\Sigma (\sigma(K(X))$ is sober, then $\Sigma O(X)$ is sober.}
\end{theorem}

\begin{proof} Let $\mathcal{U}\in\sigma(K(X))$. Set $$U_{\mathcal{U}}=\{x\in X\mid \up x\in\mathcal{U}\}.$$

$\mathbf{Claim}\ 1$: $U_{\mathcal{U}}$ is open in $X$.

Suppose $\{x_{n}\}_{n\in \mathds{N}}$ is a sequence converging to $x\in U_{\mathcal{U}}$. For every $n$, take $$K_{n}=\up\{x_{k}\mid n\leq k, k\in\mathbb{N}\}\cup\up x.$$
One can check that every $K_{n}\in K(X)$ and $\up x\subseteq\bigcap\limits_{n\in\mathbb{N}}K_{n}$. Assume that there is a $$y\in(\bigcap\limits_{n\in\mathbb{N}}K_{n})\setminus\up x.$$ Then $X\setminus\dn y$ is an open neighbourhood of $x$. As $\{x_{n}\}_{n\in \mathds{N}}$ converges to $x\in X$, there exists $n_{0}$ such that every $x_{k}$ ($n_{0}\leq k$) belongs to $X\setminus\dn y$. Equivalently, $K_{n_{0}}\subseteq X\setminus\dn y$. This contradicts with $y\in\bigcap\limits_{n\in\mathbb{N}}K_{n}$. So we have $$\up x=\bigcap\limits_{n\in\mathbb{N}}K_{n}.$$ Since $\mathcal{U}$ is Scott open in $K(X)$ and
$$\bigcap\limits_{n\in\mathbb{N}}K_{n}=\bigvee\limits_{K(X)}\{K_{n}\mid n\in\mathbb{N}\}=\up x\in\mathcal{U},$$ there is $K_{n_{1}}\in\mathcal{U}$. Since $\mathcal{U}$ is an upper set in $K(X)$ and $\up x_{k}\subseteq K_{n_{1}}$ ($n_{1}\leq k$), we have $$\{\up x_{k}\mid n_{1}\leq k\}\subseteq \mathcal{U}.$$ In other words, $\{x_{k}\mid n_{1}\leq k\}\subseteq U_{\mathcal{U}}$ and hence $U_{\mathcal{U}}$ is sequential open in $X$. It follows from $X$ a sequential space that $U_{\mathcal{U}}\in\mathcal{O}(X)$.

Define a mapping $f:\mathcal{O}(X)\longrightarrow \sigma(K(X))$ by $$\forall\ U\in\mathcal{O}(X),\ f(U)=\Box U$$ and a mapping $g:\sigma(K(X))\longrightarrow\mathcal{O}(X)$ by $$\forall\ \mathcal{U}\in\mathcal{O}(X),\ g(\mathcal{U})=U_{\mathcal{U}}.$$
By Claim 1, $g$ is well-defined. One can check that $f$ and $g$ are Scott continuous and $g\circ f=id$. So $\Sigma O(X)$ is a retract of $\Sigma (\sigma(K(X))$. By Lemma 4.10, $\Sigma O(X)$ is sober. \end{proof}

\end{document}